\documentclass[12pt]{article}
\usepackage{amsmath}
\usepackage{amssymb}
\usepackage{tikz}
\usepackage{nicematrix}
\usetikzlibrary{decorations.markings}

\def\0{{\bf 0}}
\def\1{{\bf 1}}
\def\cF{{\mathcal{F}}}

\def\proof{\noindent{\bf Proof: }}
\def\qed{ \hskip 20pt{\vrule height7pt width6pt depth0pt}\hfil}
\def\forb{{\mathrm{forb}}}
\def\ext{{\mathrm{ext}}}
\def\Av{{\mathrm{Avoid}}}
\newcommand{\linelessfrac}[2]{\genfrac{}{}{0pt}{}{#1}{#2}}

\newcommand{\ncols}[1]{\| #1 \|}
\newcommand{\rf}[1]{(\ref{#1})}
\newcommand{\trf}[1]{Theorem~\ref{#1}}

\newcommand{\srf}[1]{Section~\ref{#1}}

\newtheorem{thm}{Theorem}[section]
\newtheorem{lemm}[thm]{Lemma}
\newtheorem{prop}[thm]{Proposition}
\newtheorem{cor}[thm]{Corollary}

\newtheorem{defin}[thm]{Definition}
\newtheorem{remark}{Remark}
\newtheorem{obs}[thm]{Observation}

\newtheorem{claim}{Claim}
\newenvironment{E}{\begin{equation}}{\end{equation}}
\usepackage[english]{babel}
\makeatother

\usepackage{graphicx} 
\title{Stability Theorems for Forbidden Configurations}
\author{
R.P., Anstee\thanks{Research supported in part by
NSERC}, Jaehwan Seok\thanks{Research supported in part by NSERC USRA} \\ Mathematics Department\\The University of British Columbia\\Vancouver,
B.C. Canada V6T 1Z2\\ {\small{\texttt{anstee@math.ubc.ca}}}, {\small{\texttt{jseok627@student.ubc.ca}}} 
\and Benjamin Kreiswirth,  Bowen Li \\ Budapest Semesters in Mathematics\\Budapest, Hungary \\ {\small{\texttt{benjamin.kreiswirth@gmail.com}}},\ {\small{\texttt{bowenli.math@gmail.com}}}
\and Attila Sali\thanks{Research of the second author was partially supported by the
    National Research, Development and Innovation Office (NKFIH)
    grants K--116769 and SNN-135643. 
}\\ HUN-REN Alfr\'ed R\' enyi Institute of Mathematics\\Budapest, Hungary and \\ Department of Computer Science\\ Budapest University of Technology and Economics\\ {\small{\texttt{sali.attila@renyi.hu}}} 
}
\begin{document}

\maketitle
\begin{abstract}
    Stability is a well investigated concept in extremal combinatorics. The main idea is that if some object is close in size to an extremal object, then it retains the structure of the extremal construction. In the present paper we study stability in the context of forbidden configurations. $(0,1)$-matrix $F$ is a configuration in a $(0,1)$-matrix $A$ if $F$ is a row and columns permutation of a submatrix of $A$. $\mathrm{Avoid}(m,F)$ denotes the set of $m$-rowed $(0,1)$-matrices with pairwise distinct columns without configuration $F$, $\mathrm{forb}(m,F)$ is the largest number of columns of a matrix in $\mathrm{Avoid}(m,F)$, while $\mathrm{ext}(m,F)$ is the set of matrices in $\mathrm{Avoid}(m,F)$ of size $\mathrm{forb}(m,F)$. We show cases (i) when each element of $\mathrm{Avoid}(m,F)$ have the structure of element(s) in $\mathrm{ext}(m,F)$, (ii) $\mathrm{forb}(m,F)=\Theta(m^2)$ and the size of $A\in \mathrm{Avoid}(m,F)$ deviates from $\mathrm{forb}(m,F)$ by a linear amount, or (iii) $\mathrm{forb}(m,F)=\Theta(m)$ and the size of $A$ is smaller by a constant, then the structure of $A$ is same as the structure of a matrix in $\mathrm{ext}(m,F)$.
\end{abstract}
\section{Introduction}
An $m\times n$ matrix $A$ is said to be \emph{simple} if it is a (0,1)-matrix with no repeated columns. There is a natural correspondence between columns of $A$ and subsets of $[m]$. We consider an extremal set problem in matrix terminology as follows. Let $\ncols{A}$ be the number of columns of $A$. For a given matrix $F$, we say $F$ is a \emph{configuration} in $A$ denoted $F\prec A$ if there is a submatrix of $A$ which is a row and column permutation of $F$. Define 
$$\Av(m,F)=\left\{A\,|\,A \hbox{ is }m\hbox{-rowed}, F\not\prec A\right\},$$ 
$$\forb(m,F)=\max_{A\in\Av(m,F)}\ncols{A}.$$
A matrix $A\in\Av(m,F)$ is called {\emph{extremal}} if $\ncols{A}=\forb(m,F)$ and let 
$$\ext(m,F)=\{A\in\Av(m,F)\,|\,\ncols{A}=\forb(m,F)\}.$$ 
Problems in extremal combinatorics  are first concerned with bounds but considerations of  stability are one of the next topics to explore.  In the study of forbidden configurations, we consider $A\in\Av(m,F)$
 (or $\Av(m,{\cal F})$) with $\ncols{A}$ sufficiently close to $\forb(m,F)$, and hope to  deduce that the structure of $A$ is similar to the structure of matrices in $\ext(m,F)$.

The paper seeks structural characterizations of  extremal and near extremal matrices. A number of results have already been proven and in these cases the paper is more expository in nature. We  find some \emph{strong stability} results, namely when  $A\in\Av(m,F)$ and $\ncols{A}$ is close to $\forb(m,F)$, then some structures of the  matrices in $\ext(m,F)$ appear in $A$. We ignore row and column permutations of our matrices unless explicitly stated.

Some important matrices include $I_k$, the $k\times k$ identity matrix, $T_k$ the $k\times k$ triangular matrix (with $1's$ in position $i,j$ if $i\le j$), and $K_k$ the $k\times 2^k$ matrix of all possible (0,1)-columns on $k$ rows, furthermore $K_k^{\ell}$ the $k$-rowed matrix with all possible columns of $\ell$ 1's. We denote by $F^c$ the (0,1)-complement of $F$ so that $I_k^c$ is the complement of the identity. For example, $K_k^1=I_k$ and $K_k^{k-1}=I_k^c$.  We define $\1_k$ as $k\times 1$ column of 1's, $\0_k$ as $k\times 1$ column of 0's and $\1_k\0_{\ell}$ as $(k+\ell)\times 1$ columns with $k$ 1's on top of $\ell$ 0's. Use the following notation for 2-rowed configurations: $$F(r,p,q,s)=\left[\begin{array}{@{}c@{}}\\ \\ \end{array}\right. \overbrace{
\begin{array}{@{}c}00\cdots 0\\ 00\cdots 0\\ \end{array}}^r \overbrace{
\begin{array}{c}11\cdots 1\\ 00\cdots 0\\ \end{array}}^p \overbrace{
\begin{array}{c}00\cdots 0\\ 11\cdots 1\\\end{array}}^q 
\overbrace{
\begin{array}{c@{}}11\cdots 1\\ 11\cdots 1\\\end{array}}^s \left.\begin{array}{@{}c@{}}\\ \\ \end{array}\right].$$ 
For example $F(1,1,1,1)=K_2$. 

For two simple matrices $A,B$ where $A$ is $m_1$-rowed and $B$ is $m_2$-rowed, define the product $A\times B$  as the $(m_1+m_2)$-rowed matrix of $\ncols{A}\ncols{B}$ columns consisting of each column of $A$ on top of each column of $B$. For example $K_k=[0\,1]\times K_{k-1}.$ 


\vskip 10pt
The following example demonstrates stability. The configuration $C_k$ is the $k\times k$ vertex-edge incidence matrix of the cycle of length $k$. Now $\forb(m,C_3)=\forb(m,K_3)$.  A matrix is \emph{totally balanced} if it avoids 
$\{C_3,C_4,C_5,\ldots\}$. 

\begin{thm}\cite{A80} $\ext(m,C_3)=\ext(m,\{C_3,C_4,C_5,\ldots\})$. \end{thm}
The structure of matrices in $\ext(m,C_3)$ is given in \cite{A80} including that the columns of sum 2 can be interpreted as a spanning tree.
\begin{thm}\cite{A83} Assume $A\in\Av(m,\{C_3,C_4,C_5,\ldots\})$. Then there exists an $A'\in\ext(m,\{C_3,C_4,C_5,\ldots\})$ with $A\prec A'$ \end{thm}

Thus any totally balanced matrix can be extended, by adding  further columns to an extremal totally balanced matrix.  A powerful stability.

Use the notation $F_{a,b,c,d}$ for the $(a+b+c+d)\times 2$ configuration of $a$ rows $[1\,1]$, $b$ rows $[1\,0]$, $c$ rows $[0\,1]$, $d$ rows $[0\,0]$.

\vskip 10pt
The $4\times 2$ configuration $F_{0,2,2,0}$ is a good example of stability. The structure of all matrices in $\Av(m,F_{0,2,2,0})$  were determined in \cite{ABS} before  determining the bound.  Moreover \cite{AK} has strong stability results for $F_{0,b,b,0}$ instrumental in proving asymptotic bounds for $F_{a,b,c,d}$.

Another similarly powerful stability result considers $F=F(0,1,2,0)$ and its transpose $F^T=\left[\begin{smallmatrix}
    1&0\\0&1\\0&1
\end{smallmatrix}\right]$. 

\begin{thm}\label{thm:F-FT-char} Let $F=F(0,1,2,0)$. Then
    $A \in \Av(m,F)$ or $A \in \Av(m,F^T)$ if and only if $A$ is a subset of the columns of a matrix of the following block form:
    \begin{align}\label{blockform}
        \begin{bmatrix}
            1 & B_1 & 0 & 0 & 0 & ... & 0 & 0 \\
            1 & 1 & 1 & B_2 & 0 & ... & 0 & 0 \\
            1 & 1 & 1 & 1 & 1 & ... & 0 & 0 \\
            ... & ... & ... & ... & ... & ... & ... & ... \\
            1 & 1 & 1 & 1 & 1 & ... & B_k & 0 \\
        \end{bmatrix}
    \end{align}
    where the $B_i$ are either "empty" (their columns do not exist), or are some $I_m$ or $I_m^c$ for some $m \geq 2$ (and the $B_i$ may contain additional repeated rows when $A \in \Av(m,F) \setminus \Av(m,F^T)$).  \qed
\end{thm}

\vskip 10pt

The following two results in  Section~\ref{sec:F2}  are for $\ncols{A}$  within a linear amount of the bound.
    Let
\begin{E}F_2=\left[\begin{array}{cccc} 1&1&0&0\\ 1&0&1&0\\ 0&1&0&1\\ \end{array}\right].\label{F2matrix}\end{E}
Note $\forb(m,F_2)=\lfloor m^2/4\rfloor+m+1=\forb(m,F(1,2,2,1))$ \cite{survey}.
\begin{thm}
Let $A\in\Av(m,F_2)$ such that $\ncols{A}\ge\forb(m,F_2)-m/6+c$. Then $A\prec [\0_k\,\,T_k]\times [\0_{m-k}\,\,T_{m-k}]$ where  $k,m-k \ge m/2-{\sqrt{m/6}}$.\qed \label{strong}\end{thm}
\begin{thm}\label{thm:F1221stab}
Let $m \geq 8$ and $A\in\Av(m,F(1,2,2,1)$ with  $\ncols{A} \geq (\lfloor \frac{m^2}{4} \rfloor + m + 1) - (\lfloor \frac{m}{2} \rfloor - 2)$, then $A \prec [\1_k\,\,I^c_k] \times [\0_{m-k}\,\,I_{m-k}]$ for some $k$.\qed
    \end{thm}

Proofs are in \srf{sec:F2}.
\vskip 10pt

\srf{sec:ext-1} contains results when moving one away from the extremal value we can still characterize the matrices in $\Av(m,F)$ for some $F$.
 The extremal matrices $A\in\ext(m, F(0, 1, 2, 0))$ and $A\in\ext(m, F (0, 1, 3, 0))$ were characterized in \cite{extremalpaper} using graph theory. The ideas  
provide a modest stability result characterizing  properties of $A\in\Av(m, F (0, 1, 3, 0))$ when $\ncols{A}$ is within 1 of the bound.

Analogously,  $\Av(m,2\cdot[\1_3\,\,\1_2\0_1])$ is dominated by
$\ext(2\cdot\,\1_3)=\ext(m,\1_4)$ when $\ncols{A}$ is within 1 of the bound.
\begin{thm}$\ext(m,2\cdot[\1_3\,\,\1_2\0_1])=\ext(m,\1_4)=[K_m^3K_m^2K_m^1K_m^0]$.
 Moreover if $A\in\Av(m,2\cdot[\1_3\,\,\1_2\0_1])$ and 
 $\ncols{A}\ge \forb(m,\1_4)-1$, then $A\prec [K_m^3K_m^2K_m^1K_m^0]$.\qed\label{2(13 1201)}\end{thm}
The special case $k=2$ of the foremost theorem in forbidden configurations, Theorem~\ref{thm:sauer} is used in Section~\ref{sec:F-FT}.
\begin{thm}\label{thm:sauer}\cite{sauer,shelah,vapnik}
    \[
    \forb(m,K_k)=\binom{m}{k-1}+\binom{m}{k-2}+\ldots+\binom{m}{0}.\qed
    \]
\end{thm}
Interestingly characterizing $\ext(m,K_k)$ is  hard \cite{survey}. 
 It was proven \cite{anstee1997sperner} 
    for any $m,k,t$ with $t+ k- 1 \le m$ and $k \ge 1$, that $\ext(m,K_k)$ even contains a matrix $A(m,k,t)$ with column sums $t, t + l, t + 2 .\ldots , t + k - 1$.

The most useful proof technique in forbidden configurations is called \emph{standard induction} \cite{survey}. Choose a row $r\in[m]$. Assume $A\in\Av(m,F)$. Permute the rows and columns of $A$ so that $r$ becomes the first row. After deleting row $r$ there may be repeated columns which we place in $C_r$ in the following \emph{standard decomposition} of $A$: \begin{equation} A = \begin{array}{c}
    r\rightarrow \\ \\
\end{array}
\left[\begin{array}{c@{}c@{}cc@{}c@{}c}
   0 & \cdots & 0 & 1 & \cdots & 1\\
   B_r & & C_r & C_r & & D_r\\ \end{array}\right], \label{standard}\end{equation} \noindent where $B_r$ are the columns that appear with a $0$ in row $r$ but don't appear with a $1$, and $D_r$ are the columns that appear with a $1$ but not a $0$. We note both $[B_rC_rD_r]$ and $C_r$ are simple $(m-1)$-rowed matrices. If we assume $A\in\Av(m,F )$, then $[B_rC_rD_r]\in\Av(m-1,F )$ and
   \begin{equation}\ncols{A} =\ncols{[B_rC_rD_r]}+\ncols{C_r}\le \forb(m-1,{ F})+\ncols{C_r} .\label{induction}\end{equation}
This means any upper bound on $\ncols{C_r}$ (as a function of $m$) automatically yields an upper bound on $\forb(m,{F})$ by induction.
Of course $C_r\in\Av(m-1,F)$ but more is true. Define the \emph{inductive children} of $F$ as the minimal set of configurations $\cF'$ which must be avoided in $C_r$. Potential candidates for an inductive child would be configurations $F'$ such that $[0\,1]\times F'$ cannot appear in $A$, i.e. $F\prec [0\,1]\times F'$. We ask for $\cF'$ to be \emph{minimal} to avoid having an unwieldly set. We require that if we have two configurations $F',F''$ with $F'\prec F''$ then $F''\notin \cF'$. With this definition, $C_r\in\Av(m-1,\cF')$ and $\ncols{C_r}\le \forb(m-1,\cF')$.

\section{$\mathbf{F=} \begin{bmatrix} 0 & 0 & 1 \\ 1 & 1 & 0 \\ \end{bmatrix}$, $\mathbf{F^T  =} \begin{bmatrix} 0 & 1 \\ 0 & 1 \\ 1 & 0 \end{bmatrix}$}\label{sec:F-FT}

Typically, the transpose of a matrix does not act the same way as the matrix itself, the values of $\mathrm{forb}(F)$, $\mathrm{forb}(F^T)$ and the extremal elements $\mathrm{ext}(F)$, $\mathrm{ext}(F^T)$ differ. 
However, in this particular case, the transpose $F^T$ acts extremely similarly to $F$. Note that $F=F(0,1,2,0)$.

The following proposition is  known \cite{survey}.
\begin{prop}\label{prop:noI2}
       If $A \in \Av(m,I_2)$, then the columns must be well-ordered, so:
    \begin{align}\label{eq:noI2}
        A \prec \begin{bmatrix} 1 & 0 & 0 & ... & 0 &0 \\
        1 & 1 & 0 & ... & 0 & 0\\
        1 & 1 & 1 & ... & 0 & 0\\
        1 & 1 & 1 & ... & 1 &0\\ \end{bmatrix}.
    \end{align}
    
\end{prop}
We start with $F^T$, because it is simpler.
\begin{lemm}\label{lem:FT-block}
    Let $A \in \Av(F^T)$ and $A \not \in \Av(I_2)$. Then appropriately permuting rows and columns, $A = \left[\begin{smallmatrix} A_1 & 0 & 0 \\
1 & B & 0 \\
1 & 1 & A_2 \\ \end{smallmatrix}\right]$, where the simple matrices $A_1, A_2 \in \Av(F^T)$, and $B = I_m$ or $I_m^c$ for some $m \geq 2$.
\end{lemm}

\begin{proof}
    There is some configuration of $I_2$ in $A$. So there must be some configuration of $I_m$ or $I_m^c$, $m \geq 2$, with a maximum value of $m$ (number of columns/rows). Without loss of generality let it be an $I_m$ not an $I_m^c$ (the rest of the proof would otherwise follow symmetrically replacing $0$'s with $1$'s and vice versa). Let $B$ be this configuration. Let $S$ denote the columns in $B$, and $R$ denote the rows in $B$.

    \vspace {0.1 in}
    \noindent
    Consider any two columns in $S$, and assume for the sake of contradiction they differ on any row not in $R$. Then there would be a $\begin{smallmatrix} 0 & 1 \end{smallmatrix}$ on this row, as well as a $\begin{smallmatrix} 1 & 0 \\ 0 & 1 \end{smallmatrix}$ on the two rows in $R$ where these columns have their $1$ entry (shown below). This forms a configuration $F^T$, a contradiction. 
    \begin{align*}
        \begin{bmatrix}
            & \begin{matrix} 0 & 1 & \hspace{0.5 in} \end{matrix}   \\
            & ... & \\
            & \begin{bmatrix} 1 & 0 &  \\
            0 & 1 &  \\
             &  & \ddots \\ \end{bmatrix}  \\
        \end{bmatrix}
    \end{align*}
    \noindent
    So we can split the rows not in $R$ into two categories, those where every column in $S$ has a $0$ entry (call this $R_0$), and those where every column in $S$ has a $1$ entry (call this $R_1$).

    \vspace{0.1 in}
    \noindent
    Consider a column not in $S$, and its entries on the rows in $R$. There are four cases:
    \begin{enumerate}
        \item There are multiple $1$'s and a $0$ among these entries. Then we get a configuration $F^T$ with a complementing column in $S$ (shown first below).
        \item There is exactly one $1$ among these entries. For $A$ to be simple, it must differ from the column in $S$ with a $1$ in the same place, somewhere outside of $R$. Then it forms $F^T$ with any other column in $S$. (Shown second below is the case when the difference of the two columns consists of an entry 0 in the lower left corner. The case when the difference is an entry 1 in the upper left corner is analogous.)
        \item The column has all $0$'s on rows in $S$, then we say this is a column in $S_0$.
        \item The column has all $1$'s on rows in $S$, then we say this is a column in $S_1$.
    \end{enumerate}
    So our matrix can be written in the third form below.
    \begin{align*}
      \begin{bNiceArray}{c|cccc}
   & \Block{1-4}<\Large>{\mathbf{0}} \\
  \hline
    0 & 1 & 0 & 0 & \\
    1 & 0 & 1 & 0 & \\
    1 & 0 & 0 & 1 & \\
    & & & & \ddots \\
  \hline
& \Block{1-4}<\Large>{\mathbf{1}} \\
\end{bNiceArray} 
\hspace{0.5 in}
\begin{bNiceArray}{c|ccc}
   & \Block{1-3}<\Large>{\mathbf{0}} \\
  \hline
    1 & 1 & 0 & \\
    0 & 0 & 1 & \\
    & & & \ddots \\
  \hline
0 & \Block{1-3}<\Large>{\mathbf{1}} \\
\end{bNiceArray}
\hspace{0.5 in}
\begin{bNiceArray}{c|c|c}
& \Block{1-1}<\Large>{\mathbf{0}} & \\
\hline
\Block{1-1}<\Large>{\mathbf{1}} & \mathbf{B} & \Block{1-1}<\Large>{\mathbf{0}} \\
\hline
& \Block{1-1}<\Large>{\mathbf{1}} & \\
\end{bNiceArray}
\end{align*}

\noindent
Now consider any entry of the matrix on a row in $R_0$ and a column in $S_0$. If this entry is $1$, then including this row and column with $R$ and $S$ would give us an $I_{m+1}$ (shown first below), contradicting the maximality of $m$. So the entry must be $0$.

\noindent
Consider any entry of the matrix on a row in $R_1$ and a column in $S_1$. Assume for the sake of contradiction this entry is $0$. If $m > 2$, this forms an $F^T$ with any column in $S$ (shown second below), a contradiction. If $m = 2$, it forms an $I_3^c$ (shown third below), contradicting the maximality of $m$. So the entry must be $1$. 

\noindent
Therefore $A$ has the block configuration shown fourth below, where $A_1, A_2$ hold the rest of the entries.
\begin{align*}
    \begin{bNiceArray}{c|ccc|cc}
& \Block{1-3}<\Large>{\mathbf{0}} & & \\
& 0 & 0 & 0 & 1 & \\
\hline
\Block{3-1}<\Large>{\mathbf{1}} & 1 & 0 & 0 & 0 & \Block{3-1}<\Large>{\mathbf{0}} \\
& 0 & 1 & 0 & 0 & \\
& 0 & 0 & 1 & 0 & \\
\hline
& \Block{1-3}<\Large>{\mathbf{1}} & \\
\end{bNiceArray}
\hspace{0.4 in}
\begin{bNiceArray}{cc|cc|c}
& & \Block{1-2}<\Large>{\mathbf{0}} & \\
\hline
\Block{3-1}<\Large>{\mathbf{1}} & 1 & 1 &  & \Block{3-1}<\Large>{\mathbf{0}} \\
& 1 & 0 & ... & \\
& 1 & 0 &  & \\
\hline
& 0 & 1 & \Block{2-1}<\Large>{\mathbf{1}} & \\
& & & & \\
\end{bNiceArray} \\
\begin{bNiceArray}{cc|cc|c}
& & \Block{1-2}<\Large>{\mathbf{0}} & \\
\hline
\Block{3-1}<\Large>{\mathbf{1}} & 1 & 1 & 0 & \Block{3-1}<\Large>{\mathbf{0}} \\
& 1 & 0 & 1 & \\
\hline
& 0 & 1 & 1  & \\
& & \Block{1-2}<\Large>{\mathbf{1}} & \\
\end{bNiceArray}
\hspace{0.4 in}
\begin{bNiceArray}{c|c|c}
\mathbf{A_1} & \Block{1-1}<\Large>{\mathbf{0}} & \Block{1-1}<\Large>{\mathbf{0}} \\
\hline
\Block{1-1}<\Large>{\mathbf{1}} & \mathbf{B} & \Block{1-1}<\Large>{\mathbf{0}} \\
\hline
\Block{1-1}<\Large>{\mathbf{1}} & \Block{1-1}<\Large>{\mathbf{1}} & \mathbf{A_2} \\
\end{bNiceArray}
\end{align*}

\noindent
Note that $A_1, A_2$ must be simple since $A$ is simple and the columns in them ($S_0, S_1$ respectively) are all identical to one another on the rest of the rows. It then follows immediately from $A \in \Av(F^T)$ that $A_1, A_2 \in \Av(F^T)$. 
\end{proof}

\begin{lemm}\label{lem:F-block}
    Let $A \in \Av(F)$ and $A \not \in \Av(I_2)$. Then appropriately permuting rows and columns, $A = \left[\begin{smallmatrix} A_1 & 0 & 0 \\
1 & B & 0 \\
1 & 1 & A_2 \\ \end{smallmatrix}\right]$, where the simple matrices $A_1, A_2 \in \Av(F)$, and $B$ contains $I_m$ or $I_m^c$ for some $m \geq 2$ and all its other rows are copies of the rows in the $I_m$ or $I_m^c$.
\end{lemm}

\begin{proof}
    The proof is almost identical to that of Lemma~\ref{lem:FT-block} on $A^T$ except $A^T$ is not simple.
    
    \noindent    
    As before, we start by finding the maximum $m$ so that $I_m$ or $I_m^c$ is in $A$ (WLOG say $I_m$). Let $S$ be the columns in this $I_m$. The first part of the argument is then exactly symmetric- we split the rest of the columns into those where every row in the $I_m$ has a $0$ entry $(S_0)$ or every row in the $I_m$ has a $1$ entry $(S_1)$.

    \noindent
    Next, we let $B$ be not just this $I_m$, but include also any rows identical (on all columns) to those in the $I_m$. Then for the second argument, case 1 holds symmetrically, and case 2 follows not from the simplicity of $A$ but from the inclusion in $B$ of any copies of its rows. 

    \noindent
    And from that point onwards, the rest of the proof of Lemma~\ref{lem:FT-block} applies to $F$  symmetrically. \qed
\end{proof}

\vskip 5pt

{\noindent{\bf Proof of Theorem~\ref{thm:F-FT-char}}}.     We proceed by induction on the number of columns of $A$. 
If $A\in\Av(m,I_2)$, then it is in the form of (\ref{eq:noI2}) by Proposition~\ref{prop:noI2}. 
 Thus $A$ is a matrix of the desired form, where every block $B_i$ is empty. 

   The base case is when $A$ has one column, and therefore must be in $\Av(m,I_2)$. Assume for induction that the theorem holds for matrices up to $n$ columns.

    Consider $A$ with $n+1$ columns. If $A \in \Av(I_2)$, we are done. If $A \not \in \Av(m,I_2)$, then by Lemma~\ref{lem:FT-block} or \ref{lem:F-block}, we can write $A$ as follows
    \begin{align*}
        A = \begin{bmatrix} 
        A_1 & 0 & 0 \\
        1 & B & 0 \\
        1 & 1 & A_2 \\
        \end{bmatrix},
    \end{align*}    
    where $B$ is some $I_m$ or $I_m^c$ for $m \geq 2$, and $B$ may contain repeated rows if $A \in \Av(m,F) \setminus \Av(m,F^T)$. We also know that $A_1, A_2$ are simple and still avoid $F$ or $F^T$. Then by our inductive hypothesis, they are subsets of a matrix of our desired form, so:
    \begin{align*}
        A \prec \begin{bNiceArray}{cccccc|c|cccccc}
            1 & B_1 & 0 & ... & 0 & 0 & \Block{5-1}<\Large>{\mathbf{0}} & \Block{5-6}<\Large>{\mathbf{0}} \\
            1 & 1 & 1 & ... & 0 & 0 & & & & & & \\
            1 & 1 & 1 & ... & 0 & 0 & & & & & & \\
            ... & ... & ... & ... & ... & ... & & & & & &  \\
            1 & 1 & 1 & ... & B_k & 0 & & & & & & \\
            \hline
            \Block{1-6}<\Large>{\mathbf{1}} & & & & & & B & \Block{1-6}<\Large>{\mathbf{0}} \\
            \hline
            \Block{5-6}<\Large>{\mathbf{1}} & & & & & & \Block{5-1}<\Large>{\mathbf{1}} & 1 & B_1' & 0 & ... & 0 & 0 \\
            & & & & & & & 1 & 1 & 1 & ... & 0 & 0 \\
            & & & & & & & 1 & 1 & 1 & ... & 0 & 0 \\
            & & & & & & & ... & ... & ... & ... & ... & ... \\
            & & & & & & & 1 & 1 & 1 & ... & B_l' & 0 \\
        \end{bNiceArray}
    \end{align*}

    Note that this is also in our desired form, with a new block $B$ on the diagonal with all the existing blocks $B_1, ..., B_k$ of $A_1$ and $B_1', ..., B_l'$ of $A_2$, completing the proof of this direction.

    $\Leftarrow$: Consider any subset $A$ of a matrix in our block form. 
    
    First assume the $B_i$ have no additional repeated rows. Consider any two columns in $A$. If they are not in the same block, one will be a strict subset of the other (in whichever rows the entry on the right column is $1$, so is the entry in the left column), so $F^T$ certainly cannot be formed of these two columns. If they are in the same block, they only differ in two rows with an $I_2$, so $F^T$ cannot be formed of them either. So $A \in \Av(m,F^T)$.

    Now relax the assumption and allow the $B_i$ to have additional repeated rows. Consider any two rows in $A$. If they are not in the same block, one will be a strict subset of the other (in whichever columns the entry on the top row is $1$, so is the entry in the bottom row), so $F$ cannot be formed of these two rows. If they are in the same block, they are either identical or only differ in two columns with an $I_2$, so $F$ cannot be formed of them either. So $A \in \Av(m,F)$. \qed

\begin{thm}
    $\mathrm{forb}(m,F) = \mathrm{forb}(m,F^T) = \lfloor \frac{3m}{2} \rfloor + 1$. Then a matrix in 
    $\mathrm{ext}(m,F)$ or $ \mathrm{ext}(m,F^T)$ takes the form of \rf{blockform}. For $m$ even, all blocks $B_i=I_2$. For $m$ odd, we either have $\lfloor\frac{m}{2}\rfloor$ blocks $B_i=I_2$ and one empty block, or $\lfloor\frac{m}{2}\rfloor-1$ blocks $B_i=I_2$ and one block, say $B_j$ with $B_j=I_3$ or $B_j=I_3^c$. 
\end{thm}

\begin{proof}
    An extremal configuration $A \in \mathrm{ext}(m,F)$ or $\mathrm{ext}(m,F^T)$ is in the form \rf{blockform} with each $B_i=I_{\ell}$ or $I_{\ell}^c$. 
    If $A$ is a proper subset of the columns at the right, then simply including all the columns would give a larger configuration  so $A$ equals this entire matrix. 
    
    We know count the number of rows and columns of $A$. Let $n_i$ be the number of columns of $B_i$. Let $j$ be the number of empty $B_i$. Then the total number of rows is at least $j + \sum_{i = 1}^k n_i$ (empty $B_i$ still leave a row in their place, and other $B_i$ have at least as many rows as columns). And the total number of columns is exactly $k+j+1 + \sum_{i = 1}^k n_i$. 
    
    To maximize the number of columns given a fixed number of rows $m$, we can equivalently maximize the difference between the number of columns and the number of rows. This difference is at most $k + 1$, with equality if no $B_i$ contains repeated rows.  Each nonempty $B_i$ has $n_i \geq 2$, so  $k\le\lfloor \frac{m}{2} \rfloor$ of them. Therefore $\mathrm{forb}(m, F), \mathrm{forb}(m, F^T) \leq k + 1 + m = \lfloor \frac{3m}{2} \rfloor + 1$.

    If $m$ is even, then the only way to get equality is to have $n_i=2$ for all $i$. This means we have a unique extremal matrix.
    
    On the other hand, if $m$ is odd, then we have to have $k=\lfloor\frac{m}{2}\rfloor$ nonempty blocks. This leaves us with two possibilities. One of them is that we have $k$ blocks of size $2$ each and there is an empty block. The other one is that we have $k-1$ blocks of size $2$ and one block of size $3$, this latter one is either $I_3$ or $I_3^c$. In both cases the position of the block which is not an $I_2$ distinguishes the extremal matrices.
\end{proof}
\section{Stability within a linear amount of the bound}\label{sec:F2} 
In this section we prove Theorems~\ref{strong} and \ref{thm:F1221stab}

 \subsection{$F_2$ with a quadratic bound}
\begin{thm}\cite{small_II}
  $\forb(m,F_2)=\lfloor m^2/4\rfloor+m+1$ \qed
\label{F2bound}\end{thm}

The proof given in \cite{survey} uses shifting and Tur\'an's triangle bound. We provide a new approach while establishing $|\ext(m,F_2)|=1$.

\begin{thm}
$\ext(m,F_2)=[\0_{\lceil m/2\rceil}|\,T_{\lceil m/2\rceil}]\times     [\0_{\lfloor m/2\rfloor}\,T_{\lfloor m/2\rfloor}]$ \label{extF2}\end{thm} 
\proof Apply standard induction as in \rf{standard}. Note that $$[\0_{\lceil m/2\rceil}|\,T_{\lceil m/2\rceil}]\times     [\0_{\lfloor m/2\rfloor}\,T_{\lfloor m/2\rfloor}]\in\ext(m,F_2).$$ We wish to show uniqueness. Note that there is no row of 1's or row of 0's in $[\0_{a}\,T_{a}]$. For any row $r$ apply standard induction  to obtain $C_r$. 
$$A=\left[\begin{array}{cccc} 00\cdots 0&00\cdots 0&11\cdots 1&11\cdots1 \\ B_r&C_r&C_r&D_r \\\end{array} \right]$$ Compute $\lfloor m^2/4\rfloor+m+1-\left(\lfloor (m-1)m^2/4\rfloor+(m-1)+1\right)=\lfloor m/2\rfloor$. We use induction on $m$ to assert $\ncols{[B_rC_rD_r]}\le \lfloor (m-1)m^2/4\rfloor+(m-1)+1$ and thus if $\ncols{C_r}< \lfloor m/2\rfloor$, then $\ncols{A}<\lfloor m^2/4\rfloor+m+1$, a contradiction. We deduce that for $r\in[m]$, $\ncols{C_r}\ge \lfloor m/2\rfloor$.

Note that $F_2=[0\,1]\times I_2$ and hence the only inductive child of $F_2$ is $I_2$. Thus $C_r\in\Av(m,I_2)$ so that columns of $C_r$ are from $[\0_{m-1}\,T_{m-1}]$. Now find a set of rows $R_r\subset[m]\backslash r$ so that $C_r|_{R_r}$ is simple and $\ncols{C_r}=\ncols{C_r|_{R_r}}$. This is achieved by deleting rows from $C_r$ that don't affect simplicity e.g. delete any row of $0$'s and any row of $1$'s as well as any repeated row. Assume $|R_r|=k$. Then we can order the rows of $R_r$ so that $C_r|_{R_r}=[\0_{k}\,\,T_{k}]$.

In order that $\ncols{A}\ge \lfloor m^2/4\rfloor+m+1$, we deduce by induction that $\ncols{C_r}=|R_r|+1\ge \lfloor m/2\rfloor$. Thus in the rows of $R_r$, for any $S=\{i,j\}\in\binom{R_r}{2}$ with $i<j$, then $[\0_2\,T_2]$ is a submatrix of $A|_{\{i,j\}}$, giving us 6 columns in $ A|_{\{r,i,j\}}$. Then $\left[\linelessfrac{0}{1}\right]$ is not a submatrix of $ A|_{R_r}$ else we have 7 columns in $ A|_{\{r,i,j\}}$ yielding $F_2\prec A$. Note that what happened in $C_r$ has propogated to $B_rC_rD_r$.

Now choose $r'\in R_r$ and apply standard induction to obtain $C_{r'}$. We seek pairs of columns that differ in the single row $r'$. But in the columns of $ A|_{R_r}$, the only pairs of such columns have either rows of 0's or rows of 1's in the rows $R_r\backslash r'$. Thus $R_r\cap R_{r'}=\emptyset$.

Given that each $R_r$ is at least $m/2$, we deduce that there are only two versions, say from rows $r,r'$ with $R_r\cup R_{r'}=[m]$ while $R_r\cap R_{r'}=\emptyset$. We obtain that the columns of $A$ are contained in the cross product and if $A$ is extremal we have $A= [\0_{\lceil m/2\rceil},T_{\lceil m/2\rceil}]\times
    [\0_{\lfloor m/2\rfloor}\,T_{\lfloor m/2\rfloor}]$. \qed
\vskip 10pt
We can adapt this argument to get a strong stability result.

\noindent {\bf Proof of \trf{strong}}: The proof is a slight variation on the  proof above for \trf{extF2}. As before, standard induction \rf{standard} is used. $$m^2/4+m+1-m/6+c-(m-1)^2/4-(m-1)-1> 
m/3+c.$$ Note that means that for all $r\in [m]$, $\ncols{C_r}> m/3$ with $c=1$.
 As above, we define $R_r$ with $|R_r|=\ncols{C_r}$ and $A|_{R_r}\prec [\0_{|R_r|}\,|\,T_{|R_r|}]$. As before we deduce that for $r'\in R_r$ that $R_r\cap R_{r'}=\emptyset$. Given that $|R_r|>m/3$ for each $r\in[m]$, we deduce that there are two rows $r,r'$ with $R_r\cup R_{r'}=[m]$ yielding $A\prec [\0_k\,|\,T_k]\times [\0_{\ell}\,|\,T_{\ell}]$ where $k=|R_r|$ and $\ell=|R_{r'}|=m-r$. Note that $k-\ell$ must be `small' in order to have $\ncols{A}\ge m^2+m+1-m/6+c$ columns. A bound of $k-\ell\le{\sqrt{m/6}}$ follows. \qed
\subsection{$F=F(1,2,2,1)$}\label{sec:F1221}
We prove Theorem~\ref{thm:F1221stab} in this section. The following well known fact is used in the proof.
\begin{obs}\label{obs:trianglefree-bip}
A triangle-free graph $G$ with $m$ vertices and at least $\lfloor \frac{m^2}{4} \rfloor - (\lfloor \frac{m}{2} \rfloor - 2)$ edges is bipartite.
\end{obs}
For a large enough simple matrix $A\in\Av(m,F(1,2,2,1))$, a directed graph on $m$ vertices is defined as follows:
\begin{defin}
Given a 0-1 matrix $A$ with $m$ rows and $n$ columns ($n > m+1$), the \emph{procedure graph}  is a graph $G = (V, E)$, where each $v \in V$ represents a row in $A$ and the edge set $E$ is produced by a procedure defined as follows. 

Start with $E_0 = \emptyset$ and $A_0 = A$. Then, for each integer $i \geq 0$ such that $\ncols{A_i} > m+1$, we can find a $K_2$ in some rows $j,k$ of $A_i$ (by Theorem~\ref{thm:sauer}). Since $A_i$ (a submatrix of $A$) forbids $F$, one of $\begin{smallmatrix}
    0\\ 1
\end{smallmatrix}$ or $\begin{smallmatrix}
    1\\ 0
\end{smallmatrix}$ must occur exactly once in row $i, j$. Then, we remove the corresponding column from $A_i$ getting $A_{i+1}$, and add a directed edge $(j,k)$ or $(k, j)$ (choose the one that makes 0 go to 1) to $E_i$ obtaining $E_{i+1}$. The procedure terminates when $\ncols{A_i} \leq m+1$. As a result, the final edge set $E_{\ncols{A}-m-1}$ is the edge set $E$ for the procedure graph. 
\end{defin}
The edge set $E_{\ncols{A} - m - 1}$ are the \emph{forcing edges}. Their existence allows to use \emph{primary forcing}, that is for a forcing edge $(i,j)$ that occurs in column $\alpha$, we have entry $0$ ($1$) in row $i$ ($j$) in a column $\beta\not=\alpha$, then the $j$-entry ($i$-entry) of $\beta$ is also $0$ ($1$).   Given a partially-filled 0-1 matrix and a set of forcing edges, one can fill  several entries of the remaining matrix using primary forcing.

\begin{figure}[!htb]
    \centering
    \begin{minipage}{.5\textwidth}
        \centering
        \includegraphics[width=0.5\linewidth]{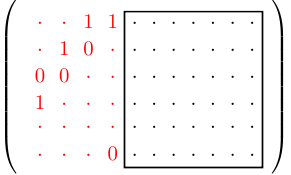}
        \caption{Before primary forcing}
        \label{fig:prob1_6_2}
    \end{minipage}%
    \begin{minipage}{0.5\textwidth}
        \centering
        \includegraphics[width=0.5\linewidth]{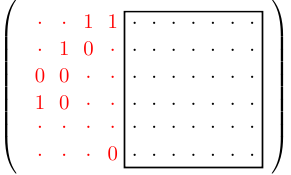}
        \caption{after one step of primary forcing}
        \label{fig:prob1_6_1}
    \end{minipage}
\end{figure}
\begin{remark}
  Given 0-1 matrix $A$, the procedure graph defined might not be unique, but all of them have the primary forcing property as shown above. 
\end{remark} 

\begin{prop}
    The underlying undirected graph of the procedure graph defined above is simple. 
\end{prop}

\begin{proof}
    For the sake of contradiction assume that the procedure graph is not simple, that is  there exists $i_1 < i_2$ such that $\{j_1, k_1\} = \{j_2, k_2\}$. 

Note that in $i_1$\textsuperscript{st} step, we removed the only $\begin{smallmatrix}
    0\\ 1
\end{smallmatrix}$ or $\begin{smallmatrix}
    1\\ 0
\end{smallmatrix}$ from the pair of rows $\{j_1, k_1\} = \{j_2, k_2\}$. Then, each $A_{\ell}$ misses one of  $\begin{smallmatrix}    0\\ 1
\end{smallmatrix}$ or $\begin{smallmatrix}
    1\\ 0
\end{smallmatrix}$ on that pair of rows for $\ell>i_1$, in particular there is no $K_2$ in $A_{i_2}$ on rows $\{j_1, k_1\} = \{j_2, k_2\}$, a contradiction. 
\end{proof}

\begin{lemm}\label{Lemma directed cycle free}
    A procedure graph is directed cycle free.
\end{lemm}

\begin{proof}
    Assume for the sake of contradiction that there is a directed cycle of vertices $v_1, v_2, ..., v_n, v_{n+1} = v_1$. Without loss of generality let $(v_n, v_1)$ be the last edge added to $G$ in the procedure. So at this stage in the procedure there is a full $K_2$ outside of the already-removed columns. And these already-removed columns contain the only $\begin{smallmatrix} 0 \\ 1 \end{smallmatrix}$s between $(v_i, v_{i+1})$, $i < n$:
    
    \begin{equation*}
    \left[
    \begin{array}{c c c c c | c c c c}
        0 & & & & &  0 & 1 & 1 & 0 \\
        1 & 0 & & & & & & & * \\
        & 1 & 0 & & & & & & * \\
        & & ... & ... & & & & &  * \\
        & & & 1 & 0 & & & &  * \\
        & & & & 1 & 0 & 0 & 1 & 1 \\
    \end{array}
    \right]
    \end{equation*}

    Let $k$ be the minimal integer so that the element in the $k$th row and last column shown above is $1$ ($k$ at most $n$, and strictly greater than $1$). Then the element above in the $k-1$st row is $0$. But then there is another $\begin{smallmatrix} 0 \\ 1 \end{smallmatrix}$ between $v_{k-1}$ and $ v_k$, $k - 1 < n$, a contradiction.
\end{proof}

\begin{defin}
    A \emph{single-edge transitive graph} is a directed cycle with exactly one edge inverted. 
\end{defin}

\begin{lemm} \label{Lemma transitivity edge free}
    A procedure graph is single edge transitivity free
\end{lemm}

\begin{proof}
    Assume for the sake of contradiction that a procedure graph $G$ has a single edge transitivity. That is, there is a directed path of vertices $v_1, v_2, ..., v_n$ in $G$, $n > 2$, as well as a directed edge from $v_1$ to $v_n$. Then there are $n$ distinct columns containing the unique $\begin{smallmatrix} 0 \\ 1 \end{smallmatrix}$s between each of these pairs of rows $(v_i, v_{i+1}), i < n$, and $(v_1, v_n)$:

    \begin{equation*}
    \left[
    \begin{array}{c c c c c c}
        0 & & & & &  0 \\
        1 & 0 & & & & * \\
        & 1 & 0 & & & * \\
        & & ... & ... & & * \\
        & & & 1 & 0 &  * \\
        & & & & 1 & 1 \\
    \end{array}
    \right]
    \end{equation*}

    Let $k$ be the minimal integer so that the element in the $k$th row and last column shown above is $1$ ($k$ at most $n$, and strictly greater than $1$). Then the element directly above in the $k-1$st row is $0$. But then this is another $\begin{smallmatrix} 0 \\ 1 \end{smallmatrix}$ between $v_{k-1}$ and $v_k$, for $k - 1 < n$, a contradiction.
\end{proof}

\begin{cor}
    A procedure graph $G$ has no undirected triangles.
\end{cor}

\begin{proof}
    An undirected triangle can take two forms as a  directed 3-cycle, or a 2-path with a single transitive edge. By Lemma~\ref{Lemma directed cycle free} the first is not possible in $G$, and by Lemma~\ref{Lemma transitivity edge free} the second is not possible in $G$. So $G$ has no undirected triangles.
\end{proof}

\begin{defin}
    A large procedure graph is a procedure graph $G$ of some $A$ where $\ncols{A} \geq \lfloor \frac{m^2}{4} \rfloor + m + 1 - (\lfloor \frac{m}{2} \rfloor - 2)$
\end{defin}
Observation~\ref{obs:trianglefree-bip} gives

\begin{lemm} \label{undirected bipartite graph}
    A large procedure graph $G$ is bipartite in an undirected sense.
\end{lemm}    

\begin{defin}
    The graph 
\begin{tikzpicture}[scale = 0.4]
           \node[circle,fill=black,inner sep=0pt,minimum size=3pt] (a) at (0,2){};
           \node[circle, fill=black, inner sep=0pt,minimum size=3pt] (b) at (2,2){};
           \node[circle, fill=black, inner sep=0pt,minimum size=3pt] (c) at (0,0){};
           \node[circle, fill=black, inner sep=0pt,minimum size=3pt] (d) at (2,0){};
           \draw[->] (a) -- (c);
           \draw[<-] (a) -- (d);
           \draw[<-] (b) -- (d);
           \draw[->] (b) -- (c);
       \end{tikzpicture}
        is called a diamond. The graph \begin{tikzpicture}[scale = 0.4]
           \node[circle,fill=black,inner sep=0pt,minimum size=3pt] (a) at (0,2){};
           \node[circle, fill=black, inner sep=0pt,minimum size=3pt] (b) at (2,2){};
           \node[circle, fill=black, inner sep=0pt,minimum size=3pt] (c) at (0,0){};
           \node[circle, fill=black, inner sep=0pt,minimum size=3pt] (d) at (2,0){};
           \draw[->] (a) -- (c);
           \draw[->] (d) -- (a);
           \draw[->] (b) -- (d);
           \draw[->] (b) -- (c);
       \end{tikzpicture} is a 4-cycle with a single transitive edge. 
\end{defin}

\begin{lemm} \label{Lemma diamond free}
    A large procedure graph is diamond-free. 
\end{lemm}

\begin{proof}
    By way of contradiction, suppose a procedure graph $G$ contains a diamond. Then, we have the following configurations with the submatrix on those 4 edgees. 
    \[
    \begin{bmatrix}
                 0&0 &  & \\
                 1&  & 0 & \\
                 & 1 &  &0 \\
                 &  &  1& 1\\
    \end{bmatrix}
    \]
    Primary forcing results in the following matrix: 
    \[
    \begin{matrix}
        \begin{matrix}
            i\\j_1\\j_2\\k
        \end{matrix}&
        \begin{bmatrix}
                 0&0&1&1 \\
                 1&0&0&1\\
                 0&1&1&0 \\
                 0&0&1& 1\\
    \end{bmatrix}
    \end{matrix}
    \]
    Consider the columns other than the ones containing the four forcing edges. Among the four rows corresponding to those four vertices (denoted as $i, j_1, j_2, k$), if the $k$\textsuperscript{th} row of a column has a 1, it must have all 1's on the four rows, by primary forcing. Similarly, if the $i$\textsuperscript{th} row has a 0 in a column, then that column must have all 0's in rows $i, j_1, j_2, k$. Since the pair $i_1, i_2$ of rows contains $\begin{bmatrix}
        1&0&0&1\\0&1&1&0
    \end{bmatrix}$, we know that we can not have both an all 1's column and an all 0's column. This means that either entries in $k$\textsuperscript{th} row in other columns are 0 or entries in $i$\textsuperscript{th} row in other columns are 1. WIthout loss of generality, we may assume that entries $k$\textsuperscript{th} row in other columns are 0. Now, if $k$\textsuperscript{th} row is removed together with the two columns($\begin{smallmatrix} 1\\0\\1\\1\end{smallmatrix}$ and $\begin{smallmatrix}
        1\\1\\0\\1
    \end{smallmatrix}$), an $m-1$-rowed simple matrix $A'\in\Av(m-1,F)$ is obtained with $\ncols{A'}>\forb(m-1,F)$, a contradiction. 
\end{proof}

\begin{lemm}\label{lemma:00-free11-free}
    If $A \in \Av(\left[\begin{smallmatrix} 0 \\ 0 \end{smallmatrix}\right])$, then $A \prec [\1I^c]$. If $A \in \Av(\left[\begin{smallmatrix} 1 \\ 1 \end{smallmatrix}\right])$, then $A \prec [\0I]$. 
\end{lemm}

\begin{proof}
    Assume that $A \in \Av(\left[\begin{smallmatrix} 0 \\ 0 \end{smallmatrix}\right])$. This means that each column has at most one zero entry. Then, $[\1I^c]$ represents all possible columns with at most one zero entry. The second statement follows from the same argument by swapping 0 and 1. 
\end{proof}

Now, we can prove Theorem~\ref{thm:F1221stab}.
\begin{thm}
    Let $m \geq 8$ and $A\in\Av(m,F)$ with $\ncols{A} \geq (\lfloor \frac{m^2}{4} \rfloor + m + 1) - (\lfloor \frac{m}{2} \rfloor - 2)$, then $A \prec [\1I^c]_k \times [\0I]_{m-k}$ for some $k$.
\end{thm}

\begin{proof}
Let $m \geq 8$ and $A\in\Av(m,F)$ with $\ncols{A} \geq (\lfloor \frac{m^2}{4} \rfloor + m + 1) - (\lfloor \frac{m}{2} \rfloor - 2)$.
By Lemma \ref{undirected bipartite graph}, the procedure graph $G$ is bipartite (in the undirected sense). Write $G = (V, W, E)$ with edges in $E$ are between $V$ and $W$ only. $|V| + |W| = m$, without loss of generality let $|V| \leq |W|$ and write $|V| = \lfloor \frac{m}{2} \rfloor - j$ and $|W| = \lceil \frac{m}{2} \rceil + j$ for $j \in \mathbb{Z}^{\geq 0}$. 

\begin{claim}
    Both $V$ and $W$ have at least 2 vertices of full degree.
\end{claim}

\begin{proof}
Assume for the sake of contradiction that at least one of $V$ or $W$ has fewer than $2$ vertices of full degree. Let $B \in \{V, W\}$ be this side, and $C \in \{V, W\}$ be the other. Then all but one vertex in $B$ have degree at most $|C| - 1$. So:

\begin{align*}
|E| &= \sum_{v \in B} \deg(v) \leq |C| + (|B| - 1) (|C| - 1) = |B| |C| - |B| + 1 \\ & \leq |W| |V| - \min(|W|, |V|) + 1 \Rightarrow \\
 |E| &\leq (\lceil \frac{m}{2} \rceil + j)(\lfloor \frac{m}{2} \rfloor - j) - (\lfloor \frac{m}{2} \rfloor - j) + 1 = \lfloor \frac{m^2}{4} \rfloor - \lfloor \frac{m}{2} \rfloor + 1 + (j - j^2)
\end{align*}

But note that $j - j^2 \leq 0$ for all possible $j$. So this implies that $\ncols{A} - (m+1) = |E| \leq \lfloor \frac{m^2}{4} \rfloor - \lfloor \frac{m}{2} \rfloor + 1$, a contradiction with our assumption on the size of $\ncols{A}$. 
\end{proof}

So we have $v_1, v_2 \in V$, $w_1, w_2 \in W$, all with full degree. This gives us an undirected four-cycle $w_1, v_1, w_2, v_2, w_1$. By Lemma \ref{Lemma directed cycle free}, this cannot be oriented as a directed four-cycle. By Lemma \ref{Lemma transitivity edge free}, this cannot be oriented as a three path and a single transitive edge. By Lemma \ref{Lemma diamond free}, this cannot be oriented as a diamond. The only remaining options are $w_1 \rightarrow v_1 \leftarrow w_2 \rightarrow v_2 \leftarrow w_1$ or $w_1 \leftarrow v_1 \rightarrow w_2 \leftarrow v_2 \rightarrow w_1$. So these edges are either all directed from $W$ to $V$ or all from $V$ to $W$. Without loss of generality say it is the former.

Consider any $v' \in V \setminus \{v_1, v_2\}$. This vertex also connects to both $w_1$ and $w_2$, so there is also an undirected four-cycle $w_1, v', w_2, v_1, w_1$. Given that $w_2 \rightarrow v_1$ and $w_1 \rightarrow v_1$  are already oriented this way, all four edges must be oriented from $W$ to $V$ to prevent diamonds or single-edge transitivity. So all of the edges between $\{w_1, w_2\}$ and $V$ will be directed from $W$ to $V$. With identical logic, we see that all of the edges between $W$ and $\{v_1, v_2\}$ will be directed from $W$ to $V$.

Consider an arbitrary pair of vertices $w_3, w_4 \in W$. They have two common descendants $v_1$ and $v_2$. Since $m \geq 8$, we have $2(m - 2) < \lfloor \frac{m^2}{4} \rfloor - \lfloor \frac{m}{2} \rfloor + 2$ the number of edges in $E$, so there must be more than $2$ vertices in $W$. Let $w_5 \in W \setminus \{w_3, w_4\}$, which as shown will also have $v_1$ as a descendant. Consider the columns of $A$ corresponding to these five edges in the procedure:

\begin{align*}
\begin{array}{c|c c c c c}
    w_3 & 0 & \color{red}{1} & 0 & \color{red}{1} & \color{red}{1} \\
    w_4 & \color{red}{1} & 0 & \color{red}{1} & 0 & \color{red}{1} \\
    w_5 & & & & & 0 \\
    \hline
    v_1 & 1 & 1 & & & 1 \\
    v_2 & & & 1 & 1 & \\
\end{array}
\end{align*}

Shown in black are the single $\left[\begin{smallmatrix} 0 \\ 1 \end{smallmatrix}\right]$'s each edge in $G$ corresponds to. Shown in red are entries that are forced to be $1$, to prevent a second $\left[\begin{smallmatrix} 0 \\ 1 \end{smallmatrix}\right]$ from either $w_3$ or $w_4$ to either $v_1$ or $v_2$. Then in order to forbid $F$, there cannot be a $\left[\begin{smallmatrix} 0 \\ 0 \end{smallmatrix}\right]$ anywhere in the rows corresponding to $w_3$ and $w_4$.

This applies to any pair of rows corresponding to a pair of vertices in $W$. Meanwhile, an identical argument replacing $0$'s with $1$'s in the above matrix shows that for any pair of vertices in $V$, their corresponding rows must forbid $\left[\begin{smallmatrix} 1 \\ 1 \end{smallmatrix}\right]$.

So the rows of $A$ are split into two groups corresponding to $W$ and $V$. The former group forbids $\left[\begin{smallmatrix} 0 \\ 0 \end{smallmatrix}\right]$ and the latter forbids $\left[\begin{smallmatrix} 1 \\ 1 \end{smallmatrix}\right]$. Then by Lemma~\ref{lemma:00-free11-free}, $A \prec [\1\,\,I^c]_k \times [\0\,\,I]_{m-k}$.
\end{proof}

\begin{cor}
    If $A \in \mathrm{ext}(F)$, that is $\ncols{A} = \lfloor \frac{m^2}{4} \rfloor + m + 1$, then $A = [\1I^c]_{m/2} \times [\0I]_{m/2}$.
\end{cor}

\begin{proof}
    By Theorem~\ref{thm:F1221stab}, $A \subseteq [\1_k\,\,I^c_k] \times [\0_{m-k}\,\,I_{m-k}]$ for some $k$. Then the size of the direct product satisfies:
    \begin{align*}
        (k+1)(m-k+1) = |[1I^c]_m \times [0I]_{m-k}| \geq |A| = \lfloor \frac{m^2}{4} \rfloor + m + 1 = (\lceil \frac{m}{2} \rceil + 1)(\lfloor \frac{m}{2} \rfloor + 1)
    \end{align*}

    This can only be satisfied by taking $k = m/2$ (either floor or ceiling), and letting $A$ be the entirety of the product. So $A = [1I^c]_{m/2} \times [0I]_{m/2}$.
\end{proof}

\section{Within a constant of the bound}\label{sec:ext-1}
This section has some cases where when $A\in\Av(m,F)$ with $\ncols{A}=\forb(m,F)-1$, then $A$ has some of the  structure of matrices in $\ext(m,F)$.
\subsection{$2\cdot[\1_3\,|\,\1_2\0_1]$}
Let us recall Theorem~\ref{2(13 1201)}.
\begin{thm}$\ext(m,2\cdot[\1_3\,|\,\1_2\0_1])=\ext(m,\1_4)=[K_m^3K_m^2K_m^1K_m^0]$.
 Moreover if $A\in\Av(m,2\cdot[\1_3\,|\,\1_2\0_1])$ and 
 $\ncols{A}\ge \forb(m,\1_4)-1$, then $A\prec [K_m^3K_m^2K_m^1K_m^0]$.\end{thm}
 
 \proof The inductive child of $2\cdot[\1_3\,|\,\1_2\0_1]$ is
 $2\cdot\1_2$. Thus from standard induction $C_r\in\Av(m-1,2\cdot\1_2)$. Given 
 $\ncols{A}\ge \forb(m,\1_4)-1$, we deduce that $\ncols{C_r}
 \ge \forb(m-1,\1_3)-1$. If $C_r$ has a column of sum $k$ for $k\ge 3$,  then this produces $\binom{k}{2}$ pairs of rows with 1's.  For $k\ge 3$,  $\binom{k}{2}\ge 3$. By pigeonhole argument on pairs of rows with 1's we deduce 
 $\ncols{C_r}\le\forb(m-1,\1_3)-2$, a contradiction.
 If $\ncols{A}=\forb(m,\1_4)$, then $\ncols{C_r}=\forb(m-1,\1_3)$ for all $r$. If $\ncols{C_r}=\forb(m-1,\1_3)$ for all $r$ then $C_r$ has all columns of sum 0,1,2  and hence $A=[K_m^3K_m^2K_m^1K_m^0]=\ext(m,2\cdot\1_3)=\ext(m,\1_4)$.
 
 Assume  $\ncols{C_r}=\forb(m-1,\1_3)-1$. We deduce from standard induction that $\ncols{[B_rC_rD_r]}=\forb(m-1,\1_4)$ and hence by induction on $m$, we have 
 $[B_rC_rD_r]=[K_{m-1}^3K_{m-1}^2K_{m-1}^1K_{m-1}^0]$. Now if $A$ has a column $\alpha$ of 4 or more 1's then by the structure of $C_r$ such a column must be of sum exactly 4 with a 1 in row $r$. Then since $\ncols{C_r}=\forb(m-1,\1_3)-1$ and $C_r$ has no column of sum 3, $C_r$ will  have all the columns but one of sum 0,1,2 on $m-1$ rows (not $r$). Considering  four rows $r,a,b,c$ with 1's in $\alpha$, we find a copy of $2\cdot[\1_3\,|\,\1_2\0_1]\prec A$ using column $\alpha$ and three columns of $C_r$ with 1's in row $r$. If a column $\beta$ with 1's in row $a,b$ is in $C_r$ then we only need two additional columns of $C_r$ with 1's in row $a$ and 0's in row $b$. If  column $\beta$ is not present, then choose the column with 1's in rows a,c.\qed

 \subsection{$F(0,1,3,0)$}
 
 We have  $F(0,1,3,0) = \begin{bmatrix}
       1 & 0 & 0 & 0 \\ 0 & 1 & 1 & 1
   \end{bmatrix}$.
\begin{thm}\cite{small_II}\label{thm:small_II}
   $\forb(m, F(0,1,3,0)) = \left\lfloor\frac{7}{3}m\right\rfloor + 1$. \qed
\end{thm}
Let $A\in\Av(m,F(0,1,3,0))$. We form a graph $G(A)$ on rows of $A$ of edges and directed edges as follows: \begin{itemize}
    \item $i-j$ if there are at most two $\left[\linelessfrac{0}{1}\right]$ and at most two $\left[\linelessfrac{1}{0}\right]$ submatrices in the two rows
   \item $i\rightarrow j$ if there is no submatrix $\left[\linelessfrac{0}{1}\right]$ in the two rows \end{itemize} 
   The graph $G(A)$ must have at least one of these two edges between every pair of rows to avoid $F(0,1,3,0)$.
   Assume that if we have both $i-j$ and $i\rightarrow j$, we will ignore $i-j$. We will show that in $A$ (or in some large submatrix), the undirected edges $i-j$ appear only in cliques  say $C_1,C_2,\ldots$ and moreover there is an ordering of the cliques so that the directed edges go from $C_i$ to $C_j$ for $i<j$.  This is analogous to (\ref{blockform}).

The following are proven in \cite{extremalpaper}. The idea is to consider $A\in\Av(m,F(0,1,3,0))$ with $\ncols{A}=\forb(m,F(0,1,3,0))$. If we  delete 2 rows and at most 3 columns we obtain $A'\in\Av(m-2,F(0,1,3,0))$, then we have a contradiction with $\ncols{A'}>\forb(m-2,F(0,1,3,0))$. Similarly if we delete three rows and at most 6 columns while preserving simplicity, we have a contradiction. For the case with  $\ncols{A}=\forb(m, F(0,1,3,0))-1$, we need a bit more room, while deleting 2 rows and at most 2 columns  or deleting 3 rows and at most 5 columns still yield contradictions.

\begin{lemm} Let $A\in\Av(m,F(0,1,3,0))$ with $\ncols{A}\ge\forb(m, F(0,1,3,0))-1 = \left\lfloor\frac{7}{3}m\right\rfloor$. Then 
  the following 2 cases on a triple of rows $(i,j,k)$ cannot appear in $G(A)$: 
  \begin{itemize}
    \item[a)] $i\rightarrow j$, $j\rightarrow k$, $k\rightarrow i$,
    \item[b)] $i\rightarrow j$, $j\rightarrow k$, $k-i$.
   \qed \end{itemize} 
  \label{nearopt}\end{lemm}

  \begin{proof}
  In case a), in $A$, there are no  columns  nonconstant on rows $i,j,k$ and so can delete $2$ rows and obtain a simple matrix $A'\in
  \Av(m-2,F(0,1,3,0))$ with $\ncols{A'}=\ncols{A}>\forb(m-2, F(0,1,3,0))$, a contradiction. 
  
  In case b), in $A$ the following are the possible columns 
  $$ \begin{array}{c} i\\ j\\ k\\ \end{array}
   \left[\begin{array}{ c c  c c}
        0 & 1 & 1 & 1 \\
        0 & 1 & 0 & 1  \\
        0 & 1 & 0 & 0  \\
   \end{array}\right].$$
   but with $i-k$, there are at most two columns in $A$ with $1$ on row $i$ and $0$ on row $k$. We could have repeats but only to total 2.  Thus we could  delete $2$ rows $i,j$ and the at most two columns (the only possible columns non constant on rows $i,j,k$) and obtain a simple matrix $A'\in
  \Av(m-2,F(0,1,3,0))$ with $\ncols{A'}=\ncols{A}-2>\forb(m-2, F(0,1,3,0))-1$, a contradiction.
\end{proof}\qed
\vskip 10pt
There is a third case to consider.


\vskip 10pt
\begin{lemm}
   Let $A\in \Av(m, F( 0, 1, 3, 0))$. If  a triple of vertices $(i, j, k)$ in $G(A)$ has the following  of edges: $i-j$, $j-k$, $i\rightarrow k$, then we can delete 2 rows and at most 4 columns to obtain a matrix $A'\in\Av(m-2, F(0, 1, 3, 0))$ . \label{cliquecase}\end{lemm}

\begin{proof}    
   Consider the triple of rows in $A$, say $i,j,k$ with $i-j$, $j-k$, $i\rightarrow k$. Then the possible columns
\begin{E}\begin{array}{c}
\begin{array}{ccccccc}
& & &a&b&c&d\\
\end{array}\\
\left[\begin{array}{cccccc}
0&1&0&1&1&1\\ 0&1&1&0&0&1\\0&1&0&0&1&0\\ \end{array}\right]\\ \end{array}\label{badtriple}\end{E}
We have restrictions on multiplicities $a,b,c,d$ forced by 
$i-j$, $j-k$, $i\rightarrow k$: 
$a,b,c,d\le 2$ and $a+d\le 2$, $b+c\le 2$ yielding that $a+b+c+d\le 4$.  This is true in $A$ so that there are at most 4  columns non constant on $i,j,k$.   
We can delete two rows $i,j$ and the at most 4  columns non constant on $i,j,k$ to obtain $A'\in\Av(m-2, F(0, 1, 3, 0))$. \qed\end{proof}

\vskip 10pt
This is not an immediate contradiction but if we delete $6$ rows and at most $12$ columns, this is a contradiction.

  \begin{lemm} Let $A\in\Av(m,F(0,1,3,0))$ with $\ncols{A}\ge\forb(m, F(0,1,3,0))-1 = \left\lfloor\frac{7}{3}m\right\rfloor$. There are 3 cases in $G(A)$ on a triple of rows $(i,j,k)$:
  \begin{itemize}
    \item[a)] $i\rightarrow j$, $j\rightarrow k$, $k\rightarrow i$,
    \item[b)] $i\rightarrow j$, $j\rightarrow k$, $k-i$,
    \item[c)]  $i-j$, $j-k$, $i\rightarrow k$.
    \end{itemize}.
    Then either the  3 cases a), b), c) do not occur or we can delete 2 rows and up to 4 columns so that the result $A'$ is simple and the 3 cases do not occur  or we can delete 4 rows and up to 8  columns so that the result $A''$ is simple and the 3 cases do not occur. 
   \label{clique}\end{lemm}

\begin{proof} Cases a),b) are handled in Lemma~\ref{nearopt}. If case c) occurs in $G(A)$ we can delete 2 rows and up to 4 columns with the result  $A'\in\Av(m-2,F(0,1,3,0))$. If case c) occurs in $G(A')$ we can delete 2 rows and up to 4 columns with the result  $A''\in\Av(m-4,F(0,1,3,0))$.  If case c) occurs in $G(A'')$ we can delete 2 rows and up to 4 columns with the result  $A'''\in\Av(m-6,F(0,1,3,0))$ but $\ncols{A'''}\ge \ncols{A}-12>\forb(m-6,F(0,1,3,0)) $, a contradiction.
 \qed\end{proof}
 \vskip 10pt
 
   Lemma~\ref{clique}  implies that for $A\in\Av(m,F(0,1,3,0))$ with $\ncols{A}=\forb(m,F(0,1,3,0))-1$, there is a large submatrix (on at least $m-4$ rows) for which undirected edges $i-j$ appear only in cliques  say $C_1,C_2,\ldots$. Moreover there is an ordering of the cliques so that the directed edges go from $C_i$ to $C_j$ for $i<j$.  We may assume without loss of generality that the rows of $A$ are ordered so that they confirm to the ordering of the cliques, i.e., if row $r$ is in clique $C_i$, row $s$ is in clique $C_j$ and $i<j$, then $r<s$ holds, as well. This  general structure is like that for $\Av(m,F(0,1,2,0))$. 
   
   The possible clique sizes in $G(A)$ are restricted although cliques of size 3 will dominate.   If $\ncols{A}$ is within one of the bound, then the undirected edges may form cliques only of sizes in $\{1,2,3,\ldots ,8\}$ .  A counting argument with $\ncols{A}=\forb(m, F(0,1,3,0))-1 $ yields the possibilities.

   \begin{prop}\label{prop:column-clique}
       For any column $\alpha$ of $A$ there exists at most one clique $C_i$ that $\alpha$ is non-constant on the rows of $C_i$, furthermore $\alpha$ is $1$ in any row above $C_i$ and $0$ in any row below $C_i$.
   \end{prop}
\begin{proof}
    Assume that $\alpha$ is non-constant on the rows of $C_i$. Let $r$ be a row in a clique $C_j$ with $j<i$. There exists a row $s$ in $C_i$ such that $\alpha(s)=1$, hence $r\rightarrow s$ implies $\alpha(r)=1$, as well. Also, there exists a row $t$ in $C_i$ such that $\alpha(t)=0$. If $C_k$ is a clique with $i<k$ and $u$ is a row in $C_k$, then $t\rightarrow u$ implies that $\alpha(u)=0$.\qed
\end{proof}

Proposition~\ref{prop:column-clique} allows us to assign columns to cliques, as follows. Clique $C_i$ is assigned every column that is non-constant on $C_i$, furthermore one more possible column that is constant $1$ on $C_i$ and above and constant $0$ below it.
On a clique of size $k\ge 3$, there are $2k(k-1)$ allowed pairs of    $\linelessfrac{0}{1}$ and $\linelessfrac{1}{0}$. Each non-constant column contributes minimum $k-1$ pairs, hence there are at most $2k$ non-constant columns assigned to that clique and in total $2k+1$ columns. Furthermore, if $|C_i|=k\ge 3$ with $2k+1$ columns assigned, then 
  those columns on $C_i$ form  $[I_k\,I_k^c\,\,\1_k]$. In that case we could remove $k-3$ rows of $C_i$ and $2(k-3)$ columns assigned to it so what remains on rows of $C_i$ is $[I_3\,I_3^c\,\,\1_3]$, hence a simple matrix. If $k-3\ge 6$ this contradicts to the bound of Theorem~\ref{thm:small_II}. On the other hand, if $|C_i|=k$ with at most $2k$ columns assigned, then we could remove the whole $C_i$ 
with the columns assigned to obtain a simple matrix, which is a contradiction again, as long as $k\ge 6$. 
Cliques of size 2 have to be treated separately, since there are only two possible non-constant columns, so a clique of size 2 has at most 3 columns assigned.
The following is an easy construction starting from some $A\in\Av(m,F(0,1,3,0))$ with $\ncols{A}=\lfloor\frac{7}{3}m\rfloor$. Imagine adding $p$ cliques of size 3 before $A$ and $q$ cliques of size 3 after $A$  to obtain a matrix with $3p+m+3q$ rows as shown below. 
\begin{prop}\label{prop:add-3-cliques}
    Let $A\in\Av(m,F(0,1,3,0))$ with $\ncols{A}=\lfloor\frac{7}{3}m\rfloor$. Let $T=[I_3\,I_3^c\,\,\1_3]$. Then 
   \begin{equation}\label{eq:block-ext}
     A'=  \begin{bNiceArray}{cccc|c|cccc}
       T&1&\ldots&1&\Block{4-1}<\Large>\1&\Block{4-4}<\Large>\1& & & \\
       0&T&\ldots&1&& & & & \\
       \vdots&\vdots&\ddots&\vdots& & & & & \\
       0&0&\ldots &T& & & & & \\
       \hline
       \Block{1-4}<\Large>\0& & & &A&\Block{1-4}<\Large>\1& & & \\
       \hline
       \Block{4-4}<\Large>\0& & & &\Block{4-1}<\Large>\0&T&1&\ldots&1\\
        & & & & &0&T&\ldots&1\\
        & & & & &    \vdots&\vdots&\ddots&\vdots\\
        & & & & &  0&0&\ldots &T
       \end{bNiceArray}
   \end{equation} 
   satisfies $A'\in\Av(3p+m+3q,F(0,1,3,0))$ and $\ncols{A'}=\lfloor\frac{7}{3}(3p+m+3q)\rfloor$, as well. \hfil\qed
\end{prop}
Proposition~\ref{prop:add-3-cliques} shows that cliques of size three can freely be added to a matrix of size one off the bound. However, the possible number of cliques of other sizes is bounded. 
\begin{prop}\label{prop:cliquenumbers}
    Let $A\in\Av(m,F(0,1,3,0))$ with $\ncols{A}=\lfloor\frac{7}{3}m\rfloor$. Then 
    \begin{enumerate}
        \item there exists at most one clique each of sizes in $\{1,2,6,7,8\}$,
        \item there exist at most two cliques of size $5$,
        \item there exist at most $5$  cliques of size $4$.
    \end{enumerate}
\end{prop}
\begin{proof}
   We have seen that a clique of size $k>3$ allows removing $k-3$ rows (resp. $k$ rows) and at most $2(k-3)$ columns (resp. at most $2k$ columns) while preserving simplicity. 
   Also, we have observed that removing $\ell\ge 6$ rows with at most $2\ell$ columns results in a contradiction. This immediately implies that there exists at most one clique of size at least 6, or at most 2 cliques of size 5  or at most 5 cliques of size 4. 
   
   A similar analysis shows that we may have at most one clique of size 2. Indeed, a clique of size 2 has at most 3 columns assigned, so if there are at least two cliques of size 2, then we could remove 4 rows with at most 6 columns, which would result in a simple matrix $A'\in\Av(m-4,F(0,1,3,0))$ with more than $\forb(m-4,F(0,1,3,0))$ columns.

   Cliques of size 1 are special. In particular, there are no non-constant columns on the (single) row of a 1-clique. However, if row $r$ is a 1-clique, then any column which has $0$ entry in $r$ has all its entries $0$ below $r$. Similarly, if a column has entry $1$ in row $r$, then all its entries are $1$ above row $r$. Thus, the only possible pair of columns distinguished by only row $r$ consists of the column $1$ above $r$ and $0$ in $r$ and below, and the column which is $1$ in $r$ and above, and $0$ below $r$. This latter column is assigned to the 1-clique on $r$.  So if there are at least two cliques of size $1$, then we could remove $2$ rows and at most $2$ columns to obtain a simple matrix,  a contradiction. \hfil \qed
 \end{proof}

The bounds in Proposition~\ref{prop:cliquenumbers} are sharp, that is all of the cases may occur.
Some examples are as follows. We have 5 cliques of size 4 with 45 columns assigned, the column $\0_m$, and say $t$ cliques of size 3 with 7 columns assigned each. Then the number of rows is $m=3t+20$, while the number of columns is $45+1+7t=\lfloor\frac{7}{3}(3t+20)\rfloor$. Similarly one may take $\0_m$, an 8-clique of 17 columns assigned, and $t$ cliques of size 3 with 7 columns assigned each. Then $m=3t+8$, the number of columns is $1+17+7t=\lfloor\frac{7}{3}(3t+8)\rfloor$. Alternatively, we may have row $1$ as a 1-clique, that is we take the columns $\mathbf{0}_m$, $\begin{bmatrix}
    1\\
    \mathbf{0}_{m-1}
\end{bmatrix}$  , and $t$ cliques of size $3$ with $7$ columns assigned each, on rows $2,3,\ldots ,3t+1$. Then $m=3t+1$, the number of columns is $7t+2=\lfloor\frac{7}{3}(3t+1)\rfloor$.
\vskip 10pt
Note that Lemma~\ref{clique} could be used successfully even for $\ncols{A}$ more than one away from the bound but more constructions appear. For example the $3\times 6$ matrix you get from (\ref{badtriple}) is two away from the bound.  The clique structure  persists in $\ext(m,F(0,1,p,0))$ for $p=4,5,6$ and perhaps for larger $p$.

\bibliographystyle{alpha}
\bibliography{bib}

\newcommand{\etalchar}[1]{$^{#1}$}
\begin{thebibliography}{AKL{\etalchar{+}}24}

\bibitem[ABS11]{ABS}
Richard Anstee, Farzin Barekat, and Attila Sali.
\newblock Small forbidden configurations {V}: Exact bounds for 4x2 cases.
\newblock {\em Studia Sci. Math. Hun.}, 48:1--22, 2011.

\bibitem[AFS01]{small_II}
Richard Anstee, Ron Ferguson, and Attila Sali.
\newblock Small forbidden configurations {II}.
\newblock {\em The Electronics Journal of Combinatorics}, 8, 2001.

\bibitem[AK06]{AK}
Richard Anstee and Peter Keevash.
\newblock Pairwise intersections and forbidden configurations.
\newblock {\em European Journal of Combinatorics}, 27:1235--1248, 2006.

\bibitem[AKL{\etalchar{+}}24]{extremalpaper}
Richard Anstee, Benjamin Kreiswirth, Bowen Li, Attila Sali, and Jaehwan Seok.
\newblock Forbidden configurations and extremal matrices.
\newblock 2024+.

\bibitem[Ans80]{A80}
Richard~P Anstee.
\newblock Properties of (0, 1)-matrices with no triangles.
\newblock {\em Journal of Combinatorial Theory, Series A}, 29(2):186--198,
  1980.

\bibitem[Ans83]{A83}
Richard Anstee.
\newblock Hypergraphs with no special cycles.
\newblock {\em Combinatorica}, 3:141--146, 1983.

\bibitem[Ans13]{survey}
Richard Anstee.
\newblock A survey of forbidden configuration results.
\newblock {\em The Electronic Journal of Combinatorics}, 20(1), 2013.

\bibitem[AS97]{anstee1997sperner}
Richard~P Anstee and Attila Sali.
\newblock Sperner families of bounded {VC}-dimension.
\newblock {\em Discrete Mathematics}, 175(1-3):13--21, 1997.

\bibitem[Sau72]{sauer}
Norbert Sauer.
\newblock On the density of families of sets.
\newblock {\em Journal of Combinatorial Theory, Series A}, 13:145--147, 1972.

\bibitem[She72]{shelah}
Saharon Shelah.
\newblock A combinatorial problem: Stability and order for models and theories
  in infinitary language.
\newblock {\em Pacific Journal of Mathematics}, 41:247--261, 1972.

\bibitem[VC71]{vapnik}
Vladimir Vapnik and Alexey Chervonenkis.
\newblock On the uniform convergence of relative frequencies of events to their
  probabilities.
\newblock {\em Theory of Probability and Its Applications}, 16:264--280, 1971.

\end{thebibliography}

\end{document}